\documentclass[11pt]{amsart}
\usepackage{times}
\usepackage[T1]{fontenc}
\usepackage{amssymb, amsthm, amsmath}
\usepackage[active]{srcltx}

\usepackage{graphicx}
\usepackage[matrix, arrow,all,cmtip,color]{xy}
\usepackage{hyperref}
\usepackage{setspace}
\usepackage[margin=.8in]{geometry}

\DeclareMathOperator{\dcl}{dcl}

\DeclareMathOperator{\tp}{tp}

\DeclareMathOperator{\cl}{cl}

\newtheorem{introtheorem}{Theorem}

\newtheorem{theorem}{Theorem}[section]

\newtheorem{claim}[theorem]{Claim}

\newtheorem{corollary}[theorem]{Corollary}

\newtheorem{fact}[theorem]{Fact}
\newtheorem{lemma}[theorem]{Lemma}

\newtheorem{proposition}[theorem]{Proposition}

\theoremstyle{definition}
\newtheorem{definition}[theorem]{Definition}

\newtheorem{remark}[theorem]{Remark}

\newcommand{\sub}{\subseteq}

\newcommand{\Qq}{{\mathbb{Q}}}

\newcommand{\CQ}{\mathcal Q}

\newcommand{\CL}{{\mathcal L}}

\newcommand{\CN}{{\mathcal N}}

\newcommand{\CR}{{\mathcal R}}
\newcommand{\CM}{{\mathcal M}}

\newcommand{\CG}{{\mathcal G}}

\newcommand{\0}{\emptyset}

\renewcommand{\phi}{\varphi}

\long\def\symbolfootnote[#1]#2{\begingroup%
\def\thefootnote{\fnsymbol{footnote}}\footnote[#1]{#2}\endgroup}

\makeatletter

\def\Ind#1#2{#1\setbox0=\hbox{$#1x$}\kern\wd0\hbox to 0pt{\hss$#1\mid$\hss}
\lower.9\ht0\hbox to 0pt{\hss$#1\smile$\hss}\kern\wd0}

\def\CNotind#1#2{#1\setbox0=\hbox{$#1x$}\kern\wd0\hbox to 0pt{\mathchardef
\nn=12854\hss$#1\nn$\kern1.4\wd0\hss}\hbox to
0pt{\hss$#1\mid$\hss}\lower.9\ht0 \hbox to
0pt{\hss$#1\smile$\hss}\kern\wd0}

\def\la{\langle}
\def\ra{\rangle}
\def\Q2{\Qq^{\sqrt{2}}}

\makeatother

\title[A theory of pairs for non-valuational structures]{A theory of pairs for non-valuational structures}

\date{\today}
\author[E. Bar-Yehuda]{Elitzur Bar-Yehuda}
\email{elitzur.by@gmail.com}
  \author[A. Hasson]{Assaf Hasson$^\dagger$}
  \thanks{$^\dagger$ Supported by ISF grant No. 181/16}
  \address{Department of mathematics\\
    Ben Gurion University of the Negev\\
    Be'er Sehva\\
    Israel} \email{hassonas@math.bgu.ac.il} \urladdr{http://www.math.bgu.ac.il/\textasciitilde hasson/}

  \date{\today}
\author[Y. Peterzil]{Ya'acov Peterzil}
\address{Department of Mathematics,
    University of Haifa, Haifa, ISRAEL} \email{kobi@math.haifa.ac.il}\urladdr{http://math.haifa.ac.il/kobi/}

\begin{document}

\maketitle
\begin{abstract} Given a weakly o-minimal structure $\CM$ and its o-minimal
completion $\bar \CM$, we first associate to $\bar \CM$ a canonical language and
then prove that $Th(\CM)$ determines $Th(\bar \CM)$. We then investigate the theory
of the pair $(\bar \CM,\CM)$ in the spirit of the theory of dense pairs of o-minimal
structures, and prove, among other results, that it is near model complete, and
every definable open subset of $\bar M^n$ is already definable in $\bar \CM$.

We give an example of a weakly o-minimal structure which interprets $\bar \CM$ and
show that it is not elementarily equivalent to any reduct of an o-minimal trace.
\end{abstract}
\section{Introduction}
An expansion $\CM$ of an ordered group is \emph{weakly o-minimal non-valuational}
(below we use ``non-valuational'' for short) if it is weakly o-minimal (every
definable subset of $M$ is a finite union of convex sets) and does not admit
 any definable non-trivial convex sub-groups. Non-valuational structures were introduced in \cite{MacMaSt} and
  more systematically studied in \cite{Wenc1} and \cite{Wenc2}. In those works Wencel showed that
   to a non-valuational structure $\CM$ one can associate an o-minimal structure $\bar \CM$, whose
   universe is $\bar M$ -- the definable Dedekind completion of $\CM$ -- and with the additional property
   that the structure which $\bar \CM$ induces on (the natural embedding of) $M$ (in $\bar \CM$) is precisely the
   structure $\CM$. Wencel called the structure $\bar \CM$ \emph{the canonical o-minimal completion of $\CM$}.
    In \cite{KerenMSc} Keren shows that $\bar \CM$ has the same definable sets as the structure $\CM^*$, whose
     atomic sets are all sets of the form $\cl_{\bar M}(S)\sub \bar M^n$ for $\CM$-definable $S\sub M^n$,
     (see Proposition \ref{same language} below).
      Both Wencel and Keren's constructions have the problem that the signatures of the resulting structures depend
       on the structure $\CM$, rather than on its signature.

In the present paper we address this problem by considering, for $A\sub M$,
structures of the form $\CM^*_A$ whose atomic sets are all sets of the form
$\cl_{\bar M}(S)$ for $S$ an $\CM$-definable set over $A$. The starting point of the
present work, and the main result of  Section \ref{prelims} is:
\begin{introtheorem}\label{signature}
    Let $\CM$ be a non-valuational structure. Then $\CM_\0^*$ and $\CM^*$ have the same definable sets.
     Moreover, if $\CM\equiv \CN$ then $\CM_\0^*\equiv \CN_\0^*$.
\end{introtheorem}

This result shows that to a non-valuational theory $T$ we can associate an o-minimal
theory $T^*$ which can be viewed as an invariant of $T$. Consequently, any of the
o-minimal properties of $T^*$ can reflect on the weakly o-minimal $T$ and vice
versa. This plays a crucial role in the proof of Theorem \ref{example} below.

 Section \ref{pairs} is dedicated to the study of the theory of the pair $\CM^P=(\CM_\0^*,\CM)$ for
 $\CM$ non-valuational, in the spirit of van den Dries' study of o-minimal dense pairs
 (see \cite{vdDriesDense}). Our main result is the following:

\begin{introtheorem}
    Let $\CM$ be  non-valuational. \begin{enumerate}
\item If $\CM\equiv \CN$ then $\CM^P\equiv \CN^P$.

We let $T^P=Th(\CM^P)$ and assume $\tilde \CN=(\CN',\CN)\models T^P$.

\item If $Y\sub (N')^n$ is $\0$-definable in $\tilde \CN$ then it can be written as
a boolean combination of sets defined by formulas of the form
\begin{equation}\label{formula} \exists x_1\cdots \exists x_k (\bigwedge_{i=1}^k
x_i\in P \, \&\, \phi(x_1,\ldots, x_k,y),\end{equation} and $\phi(x,y)$ is a formula
of the o-minimal structure $\CM'$.
 \item If $X\sub N^k$ is definable in $\tilde {\CN}$ over $A\sub N$ then $X$ is
already definable in the weakly o-minimal $\CN$. \item If $U\subseteq (N')^k$  is a
definable
    open set in $\tilde{\CN}$ then $U$ is already definable in the o-minimal structure $\CN'$.
    In particular,   $\tilde{\CN}$ has an o-minimal open core.
\end{enumerate}
\end{introtheorem}
The above results show that pairs $(\CM',\CM)$ as above fit into the setting of
recent works by Eleftherious, Gunaydin and Hieronymi (see for example \cite{EGH}) on
expansions of o-minimal structures by dense predicates.

 Non-valuational structures arise naturally
  in the study of dense pairs of o-minimal structures. Namely, if $\CM\prec \CN$ are o-minimal expansions
  of ordered groups and $M$
   is dense in $N$ then the structure induced on $M$ from $\CN$ is  non-valuational
   (weak o-minimality follows from \cite{BaiPoi} and non-valuationality is easy, see e.g., \cite{ElHaKe}).   Since every ordered group which is a reduct of a non-valuational structure, or even
     elementarily equivalent to one, is also such, a question arises whether every non-valuational structure arises
     in this manner.

   First, some terminology. A non-valuational structure $\CM$ is called an \emph{o-minimal trace} if there is
   a dense pair $\CM_0\prec \CN$
    such that $\la M_0,<\ra =\la M,<\ra$ (\i.e., the structures $\CM_0$ and $\CM$ have the same underlying
    ordered set) and the induced
    structure on $M$ in the pair $(\CN,\CM_0)$ has the same definable sets as $\CM$ (see \cite{ElHaKe} for details).
          In \cite{ElHaKe} we showed that an ordered  reduct of a non-valuational o-minimal trace
        need not be an o-minimal trace itself,  and that the class
     of reducts of o-minimal traces is not closed under elementary equivalence. In the present paper we show that
     even after closing the class of o-minimal traces under reducts and elementary equivalence we still do not cover
       all non-valautional structures:

\begin{introtheorem}\label{example}
    Let $\Q2$  be the expansion of $(\Qq,+)$ by the predicate $y<\sqrt 2 x$. Then $\Q2$ is non-valuational
    and not elementarily equivalent to a reduct of an o-minimal trace.
\end{introtheorem}

Along the way we reveal a new dividing line between two types of  non-valuational
structures:
\begin{itemize}
    \item \emph{Tight} structures (of which $\Q2$ is a typical example), in which $\CM^*$ is interpretable in $\CM$.
     These are \emph{small}
    (in the sense of \cite{vdDriesDense}), and in that respect differ significantly from o-minimal traces.
    \item Non-tight structures, whose theory resembles to a much greater extent that of o-minimal traces.
\end{itemize}

\noindent {\em This project was initiated by the M.Sc thesis of the first author at
Ben Gurion University, under the supervision of the other authors. We thank
Pantelis Eleftheriou for his helpful comments.}

\section{Preliminaries}\label{prelims}
{\em We fix a non-valuational structure $\CM$ and its definable completion $\bar
M$}. Recall that the elements of $\bar M$ are all (unique) realizations of definable
cuts in $\CM$. These will be identified here with the definable open subsets of $M$
that are bounded above and downward closed. The set $\bar M$ is equipped with
ordering by inclusion. The structure $\la M,<\ra$ is naturally embedded into $\bar
M$ via the map $a\mapsto (-\infty ,a)$, and from now on we will view $M$ as a subset
of $\bar M$. The topology on $\bar M$ and $\bar M^n$ are the order and the product
topology, respectively. We let $\cl_{\bar M}(-),
\partial_{\bar M}(-)$ denote the corresponding topological operations in $\bar M^n$.
Unless otherwise stated, all definability below refers to the structure $\CM$.

Recall that a partial function $f:M^n\to \bar M$ is said to be {\em definable} if
the set $\{(x,y)\in M^{n+1}:y< f(x)\}$ is definable. Equivalently,  the family of
cuts $\{y\in M: y<f(x)\}$, for $x\in M^n$,  is a definable family (and can be
identified with a sort in $\CM$).

 We start by collecting several useful
facts concerning the relationship of $\CM$ and  various structures on $\bar M$.
We first recall the definition of a strong cell $C\sub M^n$ from \cite{Wenc1}
\footnote{We are using Wencel's definition, in a slightly different formulation than
in \cite{MacMaSt}.} The definition will be inductive in $n$ and for the induction
step we will also associate inductively  to each strong cell $C\subseteq M^n$ its
so-called \emph{iterative convex hull} $\bar C$,  $C\sub \bar C\sub \bar M^n$.
Having defined $C$ and $\bar C$ below, we say that an $\CM$-definable function
$f:C\to \bar M$ is \emph{strongly continuous} if it extends continuously to $\bar
f:\bar C\to \bar M$, and in addition either $f(C)\subseteq M$ or $f(C)\subseteq \bar
M\setminus M$. We are now ready to state the definition:

\begin{definition}
 A set $C\subseteq M$ is {\em a strong cell} if
it is either a point, in which case $\bar C=C$, or an open convex set, in which case
$\bar C$ is defined as the convex hull of $C$ in $\bar M$.

Inductively, If $C\sub M^n$ is a strong cell (with the associated $\bar C\sub \bar
M^n$) and $f,g: C\to \bar M$ are strongly continuous with $\bar f(x)<\bar g(x)$ for
all $x\in \bar C$ (note the strong assumption here!) then $\Gamma_{f}(C)$ -- the
graph of $f$ on $C$ -- and $(f,g)_C:=\{(x,y)\in M^{n+1}: f(x)<y<g(y)\}$ {\em are
strong cells}. In the first case the iterative convex hull is defined to be the
graph of the extension $\bar f:\bar C\to \bar M$, and in the second case it is
defined to be
$$\{(x,y)\in \bar M^{n+1}: x\in \bar C\, \, \&\,\,  \bar f(x)<y<\bar g(x)\}.$$
\end{definition}

\begin{remark}\label{rem-basic cell} \begin{enumerate}
\item It is easy to verify that for each strong cell $C\sub M^n$ there exists a
homeomorphic projection $\pi_C:C\to D\sub M^k$ onto $k$ of the coordinates, $k\leq
n$, whose image is an open strong cell in $M^k$. In this case $\dim C:=k$. The
coordinate functions of $\pi_C^{-1}$ are strongly continuous on $D$.

\item  Notice that  each strong cell $C$ is a subset of $M^n$ that is definable in
$\CM$, and furthermore the various functions $f$ and $g$ in the inductive definition
of $C$ are definable in $\CM$, even though they might take values in $\bar
M\setminus M$. However, in general $\bar C\sub \bar M^n$ is not definable in $\CM$
in any obvious sense because it might not be contained in finitely many sorts in
$\CM$.
\end{enumerate}
\end{remark}

We can now describe Wencel's canonical completion $\bar \CM$, but we refine his
definition so we have a better control of parameters.

\begin{definition} Given $A\sub M$, we let $\bar \CM_A$ be the expansion of $\bar M$
by all iterative convex hulls $\bar C\sub \bar M^n$, so that $C\sub M^n$ is a strong
cell defined over $A$.

\end{definition}

It is easy to see that the order relation $<$ is an atomic relation in $\bar \CM_A$.
Since $\la M,<,+\ra$ is divisible, \cite{MacMaSt}, and $M$ is dense in $\bar M$, the
group operation extends uniquely to $\bar M$, so it is strongly continuous, and its
graph $C_{+}$ is a strong cell whose iterative convex hull is the graph of a group
operation on $\bar M$ that we still denote by $+$.

We now collect some of the main results  from \cite{Wenc2}
\begin{fact} \label{Facts} Let $\CM$ be a weakly o-minimal non-valuational structure.
\begin{enumerate}
\item Every $A$-definable set has a decomposition into finitely many strong cells,
each defined over $A$.

\item The structure $\bar \CM_M$ is o-minimal.

\item If $X\sub \bar M^n$ is definable in $\bar \CM$ then $X\cap M^n$ is definable
in $\CM$.

\end{enumerate}
\end{fact}

In \cite{KerenMSc}, the language of $\bar \CM_A$ was replaced by another one, which we
find more convenient to work with.

\begin{definition} Given $A\sub M$, and an $A$-definable set $X\sub M^n$ in $\CM$, we associate
 to $X$ a predicate symbol $\hat X$. We interpret $\hat X$ in $\bar M^n$ as the topological
 closure of $X$ in $\bar M^n$, denoted by $\cl_{\bar M}(X)$, and let $\CM^*_A$ be the expansion of $\bar M$
 by all $\hat X$, for $X\sub M^n$ definable over $A$.\end{definition}

It was proved in \cite{KerenMSc} that the structures $\bar \CM_M$ and and $\CM^*_M$
have the same definable sets. We re-prove here a more precise version. We first
prove:

\begin{lemma} \label{nice closure transition}
    If $C\subseteq M^n$ is a strong cell
    then $\cl_{\bar{M}}(C)=\cl_{\bar{M}}(\bar C)$.
\end{lemma}

\begin{proof}
      Since $C\subseteq\bar C$ it suffices to show
     that $\bar C\subseteq \hat{C}$ for every strong cell $C$.  We use induction on $n$.

    If $C\subseteq M$ the claim is obvious.  Now, suppose that $\bar C\subseteq\hat{C}$ for some strong cell
$C$ and
    let $f_1,f_2<f_3$ be strongly continuous such that the range of $f_1$ is in $M$.  We let $C_1=\Gamma(f_1)_C$
      and  $C_2=(f_2,f_3)_C$ be the associated strong cells, and will show that
      $\bar C_1\subseteq\hat{C_1}$ and $\bar C_2\subseteq\hat{C_2}$.

    Let $(c,m)\in\bar C\times\bar{M}$.  If $(c,m)\in\bar C_1$ then $\bar{f}_1(c)=m$.  But then since
     $c\in\bar{C}$ and
    $\Gamma(f_1)_C$ is dense in $\Gamma(\bar{f}_1)_{\bar C}$ (because $C$ is dense $\bar C$) then $(c,m)$
    is a limit
    point of $f_1$ and therefore $(c,m)\in\hat {C_1}$.  If $(c,m)\in\bar C_2$ then $\bar{f}_2(c)<m<\bar{f}_3(c)$,
    and again since $c\in\bar{C}$ and $(f_2,f_3)_C$ is dense in $(\bar{f}_2,\bar{f}_3)_{\bar C}$ then $(c,m)$ is
    a limit point of $(f_2,f_3)_C$ and therefore $(c,m)\in\hat{C_2}$.
\end{proof}

We can now prove:
\begin{proposition}\label{same language} For every $A\sub M$, the (o-minimal) structures $\CM^*_A$ and $\bar \CM_A$
have the same $\0$-definable sets (so in particular the same definable sets).

\end{proposition}
\proof We first show that every atomic set in $\CM^*_A$ is $\0$-definable in $\bar
\CM_A$. So we take an $A$-definable $X\sub M^k$, and consider its closure $\hat
X\sub \bar M^k$. By Fact \ref{Facts}, $X$ can be written as the union
$\bigcup_{i=1}^k C_i$ of strong cells that are definable over $A$ in $\CM$. By Lemma
\ref{nice closure transition}, each $\bar C_i$ is dense in $\cl_{\bar M}(C_i)$. It
follows that $\cl_{\bar M}(X)=\bigcup_{i=1}^k \cl_{\bar M}(\bar C_i)$. Since each
$\bar C_i$ is $\0$-definable in $\bar \CM_A$, and the closure operation is itself
definable, it follows that $\cl_{\bar M}(X)$ is $\0$-definable in $\bar \CM_A$.

For the other inclusion, we need to see that for every  strong cell $C\sub M^n$ that
is definable over $A$, the set $\bar C$ is $\0$-definable in $\CM^*_A$. This is done
by induction on $n$.

    For $0$-cells  in $\CM$ this is clear. If $C\sub M$ is a $1$-cell then $\bar C$
    is an open interval $(a,b)$ in $\bar M$. The interval $[a,b]$ is $\0$-definable in $\CM^*_A$,
     hence so is $\bar C$.
     So we now assume that we have proved the result for all strong cells in $M^n$ and we prove it
      for strong-cells in $M^{n+1}$. Let $C\subseteq M^n$ be a strong cell defined over $A$.
       Let $f: C\to M$ be a strongly continuous function definable in $\CM$ over $A$, and let $Y$ be
        $\Gamma_f$, the graph of $f$. Then $\bar Y:=\{(x,\bar f(x)): x\in \bar C\}$.
        We have to show that $\bar Y$ is $\0$-definable in $\CM^*_A$. As $\bar f$ is continuous we get that
         $\bar Y=\cl_{\bar M}(\Gamma_f)\cap  (\bar C\times \bar M)$, which is
         $\0$-definable in $\CM^*_A$
          by the inductive hypothesis.

    Now let $f,g: C\to \bar M$ be $A$-definable strongly continuous functions in $\CM$, with $f<g$
    (unlike the above, we cannot
    assume here that they take values in $M$). We have
    to show that the iterative convex hull of $Y:=(f,g)_C$ is $\0$-definable in $\CM_A^*$. By
    definition,
    \[\bar Y=\{(x,y): x\in \bar C, \bar f(x)< y< \bar g(x)\}.\]
    Since, by induction $\bar C$ is $\0$-definable in $\CM^*_A$, it will suffice to show that $\bar f$
    (and similarly $\bar g$)
    is
    $\0$-definable in $\CM^*_A$. If $f$ is the constant function $-\infty$, then there is nothing to prove.
    So we assume this is not the case. By definition, the set
    \[F:=\{(x,y): x\in C, y<f(x)\} \]
    is $A$-definable in $\CM$. For every $c\in \bar C$ let
    \[s(c):=\sup\{y\in \bar {M}:(c,y)\in \cl_{\bar M} (F)\}.  \]
    Since $f$ is strongly continuous, $s(c)$ is well defined, and by definition it coincides with
    $f$ on $C$. Since $C$ is dense in $\bar C$ and $\bar f$ is the unique continuous extension of
     $f$ to $\bar C$, necessarily $s=\bar f$, and as $s$ is $\0$-definable in $\CM^*_A$, we are
     done.\qed

From now on we can use interchangeably the structures $\CM^*_A$ and $\bar \CM_A$.
 Notice however, that the language of $\bar \CM_A$
 depends on the specific structure $\CM$,
thus for different $\CM$ and $\CN$, even if elementarily equivalent, the structures
$\bar \CM_M$ and $\bar \CN_N$ are of different signature. One of the initial goals
of this work was to obtain a uniform signature by showing that the definable sets in
$\bar \CM_\0$ ad $\bar \CM_M$ are the same.  We need the following observations

\begin{proposition}\label{observation} \begin{enumerate}
\item Every $\0$-definable set in $\CM$ can be written as a boolean combination of
$\0$-definable sets each of which is the closure of an open $\0$-definable set. In
particular, this is true if $\CM$ is o-minimal.

\item The o-minimal structure $\CM_M^*$ eliminates quantifiers. Moreover, it is
sufficient to take as atomic relations all $\cl_{\bar M}(X)$ with $X\sub M^n$ an
open definable set.
\end{enumerate}
\end{proposition}
\proof (1) We first prove the result for an arbitrary definable open set  $X\sub  M^n$.
 Note that  $X=\cl(X)\setminus \partial(X)$
(here $\partial(X)$ is the boundary of $X$), and then that $$\partial(X)=\cl(X)\cap
\cl(M^n\setminus \cl(X)).$$ The set on the right is of the desired form, so we are
done.

For an arbitrary definable $X\sub M^n$,  we apply strong cell decomposition, so we
may assume that $X$ is a cell. Hence, $X$ is either a point or the graph of a
definable map $f$ from an open cell $C\sub M^{n-k}$ into $M^k$ (the $n-k$
coordinates need not be the first ones), and each of the coordinate functions of $f$
are strongly continuous.

 Thus it is sufficient to show that the graph of each strongly continuous $f_i:C\to M$
 is definable in the desired form. By the continuity of $f_i$, such a graph can be written as the complement
 in $C\times M$ of the open set:
 $$\{(x,y)\in C\times M:y>f_i(x)\}\cup \{(x,y)\in C\times M: y<f_i(x)\}.$$
Since each of the open sets can be defined in the required form, so is the graph of
$f_i$, and hence so is $X$.\qed

For (2), we first apply (1) to the o-minimal structure $\CM^*_M$ and reduce the
problem to definable sets $\hat X\sub \bar M^n$, which are the closure of an open
definable set $U\sub \bar M^n$. Since $M^n$ is dense in $\bar M^n$, $\cl_{\bar
M}(U)=\cl_{\bar M}(U\cap M^n)$. By fact \ref{Facts}, the set $U\cap
\bar M^n$ is definable in $\CM$ (possibly over parameters). We now apply (1). \qed \\

In the text the first part of the above proposition will be applied, mostly, when $\CM$ is, in fact, o-minimal.

\begin{lemma}\label{fibres}
    Let $C\subseteq M^{k+n}$  be a strong cell, $a\in \pi(C)$, where $\pi$ is the projection
    onto the first $k$-coordinate. Let $C_a=\{x\in M^n: (a,x)\in C\}$. Then
    \begin{enumerate}
        \item $C_a$ is a strong cell.
        \item $(\bar C)_a=\overline{ C_a}$.
    \end{enumerate}

\end{lemma}
\proof It is sufficient to prove the result for $k=1$ (and then proceed by
induction). This is straightforward from the definition of a strong cell.\qed

\begin{theorem}\label{elitzur} For every $A\sub M$, the structures $\bar \CM_M$ and $\bar \CM_A$ have
the same definable sets.

\end{theorem} \proof Absorbing $A$ to the language we, at this stage, assume that $A=\0$.
We first claim that for every
$n\in \mathbb N$, we have
\begin{equation}\label{eq3}\{Y\cap M^n:Y\sub \bar M^n \mbox{
definable in } \bar \CM_\0\}=\{Y\cap M^n:Y\sub \bar M^n \mbox{
definable in } \bar \CM_M\}.\end{equation}

Since $\bar \CM_\0$ is a reduct of $\bar \CM_M$ it is sufficient to prove the
right-to-left inclusion. We first show: For every $\0$-definable $X\sub M^n$, there
exists a $\0$-definable $Y\sub \bar M^n$ in $\bar \CM_\0$ such that $Y\cap M^n=X$.
Indeed, $X$ has a decomposition into $\0$-definable strong cells (see Fact
\ref{Facts}), and for each $\0$-definable strong cell $C_i$ we have $\bar C_i\cap
M^n=C_i$, so $Y=\bigcup_{i}\bar C_i$ is the desired set.

Now, let $Z\sub M^n$ be definable in $\bar \CM_M$. By Fact \ref{Facts}, $Z\cap M^n$
is definable in $\CM$, possibly over parameters. Hence, it is of the form $X_a$, for
some $X\sub M^{n+k}$ which is $\0$-definable in $\CM$ and $a\in M^k$. By what we
just shown, there is $Y\sub \bar M^k$ which is $\0$-definable in $\bar \CM_\0$, such
that $X=Y\cap M^{n+k}$. Hence,
$$Z=(Y\cap M^{n+k})_a=Y_a\cap M^n,$$ and $Y_a$ is definable in
$\bar \CM_\0$. This ends the proof of (\ref{eq3}).

We now make the following general observation:

\begin{lemma}\label{general-obs} Let $\la N,<\ra$ be a densely
ordered set, with $M\sub N$ a dense subset. Assume that $\CN_1$, $\CN_2$ are two
o-minimal expansions of $\la N,<\ra$ with the property that for every $n\in \mathbb
N$, we have

\begin{equation}\label{eq4}\{Y\cap M^n:Y\sub \bar M^n \mbox{
definable in } \CN_1\}=\{Y\cap M^n:Y\sub \bar M^n \mbox{ definable
in } \CN_2\}.\end{equation}

Then $\CN_1$ and $\CN_2$ have the same definable sets.
\end{lemma}
\proof It easily follows from the assumptions that we have
\begin{equation}\label{eq5} \{Y\cap M^n:Y\sub \bar M^n \mbox{ open definable in
} \CN_1\}=\{Y\cap M^n:Y\sub \bar M^n \mbox{ open definable in }
\CN_2\}.\end{equation}

By Proposition \ref{observation} (1), it is enough to know that for every open
$U\sub N^n$, the set $\cl(U)$ is definable in $\CN_1$ if and only if it is definable
in $\CN_2$. However, since $M$ is dense in $N$, it is enough to consider sets of the
form $\cl(U\cap M^n)$. By (\ref{eq5}), both collections of sets of the form $U\cap
M^n$, where $U$ is definable in either $\CN_1$ or in $\CN_2$, are the same..\qed

In order to prove Theorem \ref{elitzur}, we apply Lemma \ref{general-obs} to the structures $\bar \CM_\0$ and $\bar\CM_M$ using (\ref{eq3}).

\qed

\section{The structure $\CM_A$ and elementary extensions}

{\em Again, we let $\CM$ be a fixed non-valuational structure }.
From now on we shall work with $\CM^*_A$ rather than $\bar \CM_A$.

\subsection{The canonical completion and elementary extensions}
   Let $\CN$ be an
elementary extension of $\CM$. Every definable cut $C$ in $\CM$ has a natural
realization $C(\CN)$  in $\CN$ and so $\bar M$ can be embedded into $\bar N$. Under
this embedding, if $n\in \bar N$ is the supremum of a cut in $N$ which is definable
over some $A\sub M$ then $n$ is already in $\bar M$. We have:
\[
  \xymatrix{
  \CN \ar[r]^\iota \ar@{}[d]|-*[@]{\curlyvee} & \bar \CN \\ \CM \ar[r]^\iota
     & \bar \CM\ar[u]\\
  }
\]

Where $\iota$ is the natural embedding of $(M,<)$ in $(\bar M, <)$.
We now fix an arbitrary $A\sub M$ and consider the structures $\CM^*_A$ and $
\CN^*_A$. Both structures are in the language $\CL^*_A$, and we claim that $\CM^*_A$
is a substructure of $\CN^*_A$:  Indeed, first note that for {\em a fixed} $x\in
\bar M^n$, and $\epsilon>0 $ in $M$, the set $B(x,\epsilon)\cap M^n=\{y\in M^n:
|x-y|<\epsilon\}$ is definable in $\CM$ and moreover, it is uniformly definable as
$\epsilon$ varies in $M_{>0}$ ($x$ still fixed). It easily follows that for $x\in
\bar M^n$, being in the closure of a definable $X\sub M^n$ is a first order
property. Namely, for $x\in \bar M^n$,
\[x\in \cl_{\bar M}(X(M)) \Leftrightarrow x\in \cl_{\bar N}(X(N)).\] Said
differently, $\hat X(N)\cap \bar M^n=\hat X(M)$, so $\CM^*_A$ is a substructure of
$\CN^*_A$.

Our goal is to show that $\CM^*_A$ is in fact an elementary substructure of
$\CN^*_A$.We do that in several steps.

\begin{lemma}\label{claim1.1} Assume that $A\sub M$ and that $\CM$ is $|A|^+$-saturated. If $Y\sub
\bar M^n$ is $\0$-definable in $\CM_A^*$ then $Y\cap M^n$ is $A$-definable in $\CM$.
\end{lemma}
\proof By fact \ref{Facts},  $Y\cap M^n$ is definable in $\CM$. By
the saturation assumption it is enough to show that any
automorphism of $\CM$ which fixes $A$ point-wise leaves $Y\cap
M^n$ invariant. Let $\alpha:M\to M$ be such an automorphism. We
claim that $\alpha$ has a (unique) extension to a bijection $\bar
\alpha:\bar M\to \bar M$ which is an automorphism of $\CM^*_A$.
Because $\alpha$ is an automorphism of $\CM$ it sends definable
cuts to definable cuts so extends naturally to $\bar \alpha :\bar
M\to \bar M$. The  map $\bar \alpha$ is  an order preserving
bijection so in particular continuous on $\bar M$. To see that
$\bar \alpha$ is an automorphism of $\CM^*_A$, let $X\sub M^n$ be
$A$-definable and consider its closure $\hat X$. Since $\alpha
(X)=X$, continuity implies that $\bar \alpha (\hat X)=\hat X$,
thus $\bar \alpha$ is an automorphism of $\CM^*_A$.

Since $Y$ was $\0$-definable in $\CM^*_A$ it is left invariant under $\bar \alpha$,
and because $\bar \alpha(M)=M$, we have $$\alpha(Y\cap M^n)=\bar \alpha(Y\cap
M^n)=\bar\alpha(Y)\cap \bar \alpha(M^n)=Y\cap M^n.$$ \qed

\begin{lemma} \label{claim1.2} For $A\sub M$ arbitrary,
if $\CM\prec \CN$ then $\CM^*_A\prec \CN^*_A$.
\end{lemma}
\proof First note that we may assume that $\CN$ is sufficiently saturated. Indeed,
we may consider $\CN'\succ \CN$ which is saturated enough. The above would then
imply that $\CM_A^*\prec (\CN_A')^*$ and $\CN_A^*\prec (\CN_A')^*$, from which it
follows that $\CM_A^*\prec \CN_A^*$.

By The Tarski-Vaught Criterion, it is enough to prove, for every nonempty $Y\sub
\bar N$ which is definable in $\CN^*_A$ over $\bar M$, that $Y\cap \bar M\neq
\emptyset$.

Since $\CN^*_A$ is an o-minimal expansion of a group,  $Y$ contains some element
$b\in \dcl_{\CN^*_A}(\bar M)$. So, there exists a finite tuple $a=(a_1,\ldots, a_r)$
from $\bar M$, such that $b\in \dcl_{\CN^*_A}(a)$. Each $a_i$ realizes a cut in $M$,
definable in $\CM$ over some finitely many parameters. Thus there is a finite $F\sub
M$ such that each $a_i$ realizes  a cut definable over $F$. If we now let $A'=A\cup
F\sub M\sub N$ then clearly every element in $A'$ is $\0$-definable in $\CN^*_{A'}$,
hence $b$ is in $dcl_{\CN^*_{A'}}(\0)$ so the set $(-\infty,b)$ is $\0$-definable in
$\CN^*_{A'}$. Since $\CN$ is sufficiently saturated it follows from Lemma
\ref{claim1.1} that $(-\infty,b)\cap N$ is $A'$-definable in $\CN$.

Since $\CM\prec \CN$ and $A'\sub M$ it follows, as we already noted above,  that
$b\in \bar M$, so $X\cap \bar M\neq \emptyset$. Thus $\CM^*_A\prec \CN^*_A$. \qed
\\

\noindent {\bf Note}: It only makes sense to compare $\CM^*_A$ and $\CN^*_A$ for
$A\sub M$, since otherwise the two structures do not have a common language.\\

Finally, we can now prove:

\begin{theorem}\label{intersection}  For  $A\sub M$ (with no saturation assumption),
 assume that $X\sub \bar M^n$ is $\0$-definable in
the structure $\CM^*_A$. Then $X\cap M^n$ is $A$-definable in $\CM$. In particular,
if $f:\bar M^n \to \bar M$ is $\0$-definable in $\CM^*_A$ then $f\upharpoonright
M^n:M^n\to \bar M$ is $A$-definable in $\CM$.
\end{theorem}
\proof We consider an elementary extension $\CN$ of $\CM$ that is $|A|^+$-saturated.
By Lemma \ref{claim1.2}, we have $\CM^*_A\prec \CN^*_A$ and by Lemma \ref{claim1.1},
the set $Y=X(\bar N)\cap N^n$ is definable in $\CN$ over $A$. Since $\CM\prec \CN$
we can conclude that $Y\cap M^n=X(\bar N)\cap M^n$ is also definable over $A$ in
$\CM$. It is left to see that this last set equals $X\cap M^n$. Because
$\CM^*_A\prec \CN^*_A$ we have $X(\bar N)\cap \bar M^n=X$, and therefore
$$Y\cap M^n=X(\bar N)\cap M^n=X\cap M^n.$$

For the second clause, just note that the set $\{x\in M^n:x<f(x)\}$ is the
intersection of a $\0$-definable subset of $\bar M^n$ with $M^n$. \qed

We now return to Proposition \ref{observation} and Theorem \ref{elitzur} and prove
finer results:
\begin{proposition}\label{QE} For any $A\sub M$,
\begin{enumerate} \item
 The structure $\CM^*_A$ eliminates
quantifiers. In fact every $\0$-definable set is a boolean
combination of sets of the form $\cl_{\bar M}(X)$ for $X\sub M^n$
open and definable in $\CM$ over $A$.

\item If $X\sub \bar M^n$ is $\0$-definable in $\CM^*_A$ then it
is $A$-definable in $\CM^*_\0$.
\end{enumerate}
\end{proposition}
\proof (1) We may repeat the short argument in the proof of \ref{observation} with
the additional data given by Theorem \ref{intersection}, that whenever $X\sub \bar
M^n$ is $\0$-definable in $\CM^*_A$, the set $X\cap M^n$ is $A$-definable in $\CM$.
For (2), assume that $Z$ is $\0$-definable in $M^*_A$. By (1), $Z$ is a boolean
combination of atomic sets (with no extra parameters), so it is sufficient to prove
that each atomic such set $Z$ is $A$-definable in $\CM^*_\0$.
 By the first paragraph of the proof of
  Proposition \ref{same language} $Z=\bigcup_{i=1}^k \cl_{\bar M}(\bar C_i)$ for some $A$-definable
  strong cells $C_i\subseteq M^n$. So $C_i=(D_i)_a$ for some $\0$-definable set $D_i$ nd
  $a\subseteq A$. By strong cell decomposition, each  $D_i$ is itself a finite union of $\0$-definable
  strong cells, so we may write each $C_i$ as a union of the form
  $\bigcup_j(D_{i,j})_a$, where each $D_{i,j}$ is a $\0$-definable strong cell.

  By Lemma \ref{fibres} we know that $\overline {(D_{i,j})_a}=(\overline {D_{i,j}})_a$ for every $j$. The
   right-hand side of this equation is $A$-definable in $\CM^*_\0$, and hence so is its $\bar M$-closure.
    Therefore the closure of each
   $C_i$ is a finite union of sets that are $A$-definable in $\CM_\0^*$. The conclusion follows.
\qed

Since any two elementarily equivalent structures have a common elementary extension,
we can also conclude from Lemma \ref{claim1.2}:

\begin{corollary}\label{equiv} If $\CM$ is non-valuational and $\CN\equiv \CM$ then
$\bar \CM_\0\equiv \bar \CN_\0$ (both are $\bar \CL_\0$-structures), and
$\CM^*_\0\equiv \CN^*_\0$ (as $\CL^*_\0$-structures)
\end{corollary}

Finally, we shall be using the following technical lemma:
\begin{lemma}\label{early-dcl} For every $A\sub M$, $\dcl_{\CM^*_\0}(A)\cap
M=\dcl_\CM(A).$
\end{lemma} \proof Assume that $a\in dcl_{\CM_\0^*}(A)\cap M$.
Then it follows that $a\in dcl_{\CM^*_A}(\0)$ (since each element
of $A$ is $\0$-definable in $\CM^*_A$). Hence, the interval
$(-\infty,a)\sub \bar M$ is $\0$-definable in $\CM^*_A$, so by
Theorem \ref{intersection}, the intersection of $(-\infty,a)$ with
$M$ is $A$-definable in $\CM$. Because $a\in M$, we have $a\in
dcl_\CM(A)$.

For the converse, assume that $a\in dcl_\CM(A)$ (so in particular in $M$). Thus, the
interval $(-\infty,a)$ is definable in $\CM$ over $A$ and its iterative convex hull,
the interval $\overline{(\infty,a)}\sub \bar M$, is $\0$-definable in $\CM_A$. By
Proposition \ref{observation}(2), this interval is $A$-definable in $\CM^*_\0$ so $a\in
\dcl_{\CM^*}(A)\cap M$.\qed

\section{Tight weakly o-minimal structures}

As was pointed out before, the set $\bar M$ can be viewed as a union of sorts in
$\CM$, where each sort corresponds to a $\0$-definable family of cuts in $\CM$. In
general, there might be infinitely many such sorts, but in some cases there are only
finitely many such sorts.

\subsection{Definition and basic properties}
\begin{definition}
    A non-valuational structure $\CM$  is \emph{tight} if there are finitely
     many $\0$-definable
    families of cuts in $\CM$ such that
     every definable cut belongs to one of them.
\end{definition}
Clearly, if $\CM$ is an o-minimal structure then it is (trivially) non-valuational
and tight, since the family of definable cuts is just all intervals of the form
$(-\infty,a)$, as $a$ varies in $M$.

It immediately follows that if $\CM\equiv \CN$ then $\CM$ is tight if and only if
$\CN$ is tight. Thus, we may use the term ``tight'' for $T$ as well.

\begin{proposition}\label{tight1} The structure $\CM$ is tight if and only if there are finitely many
$\0$-definable functions $f_i:M^{n_i}\to \bar M$, $i=1,\ldots, k$, such that $\bar
M\sub \bigcup_{i=1}^k Im(f_i)$.

In particular, $\CM$ is tight then the structure $\CM^*$ is interpretable in $\CM$
without parameters.
\end{proposition}
\proof  The first clause is easy to verify. For the second clause, note first that
the universe of $\bar M$ is a quotient of some $M^n$ by a definable set, and
furthermore the embedding of $M$ in this quotient (i.e. the family of cuts
$\{C_x:x\in M\}$, where $C_x=\{y<x\}$) is definable in $\CM$. It is easy to see that
the ordering on $\bar M$ is definable in $\CM$ and hence $\cl_{\bar M}(X)$ is
definable in $\CM$ for every $\CM$-definable $X\sub M^n$.\qed

\begin{remark}
    The above proof shows, in fact, that the pair $(\CM^*, \CM)$ is bi-interpretable with $\CM$, i.e., not only is $\CM^*$ interpretable in $\CM$, but so is the natural embedding of $M$ in $\bar M$.
\end{remark}

\subsection{An example of a tight structure}
 We shall now see that there are
examples of tight structures which are not o-minimal.

    Let $\Qq_{vs}=\la \mathbb Q,<,+,1,\{\lambda_{q}\}_{q\in \mathbb Q}\ra$ denote the group of rational numbers,
    viewed as an ordered vector space over itself, with a function symbol for every rational scalar.
    Let $\Q2$ be the expansion of $\Qq_{vs}$ by the relation
    $$P\!{\mbox{\tiny $\sqrt{2}$}}=\{(x,y)\in \mathbb Q^2: y<\sqrt{2} x\}.$$ We
    denote the langauge by $\CL_{\sqrt 2}$. (In \cite[Section 3]{ElHaKe} a similar expansion of $\Qq_{vs}$ by the predicate
$P_\pi$ was investigated.)

\newcommand{\Psq}{P\!{\mbox{\tiny $\sqrt{2}$}}}

The idea is to eventually identify $\Psq$ with a map $x\mapsto \sqrt 2 x$ from the
structure $\Q2$ into its canonical completion.
    Our goal is to show that $\mathrm{Th}(\Q2)$ is axiomatised by the following theory $T$:
\begin{enumerate}
    \item The ordered $\Qq$-vector space axioms.
    \item An axiom expressing the fact that $\Psq$ is ``linear'':
    $$(\forall x_1,y_1,x_2,y_2) \,( ((x_1,y_1)\in \Psq \wedge (x_2,y_2)\in \Psq )\rightarrow
    (x_1+x_2,y_1+y_2)\in \Psq).$$
 \item (Ensuring that we define the positive $\sqrt 2$): $(\exists x, y) ((x,y)\in
    \Psq\land x>0 \land y>0)$
\item For all $r\in \Qq$, such that  $r<\sqrt 2$, we have $\forall x\,(x>0
\rightarrow (x,rx) \in \Psq)$,
     and for all $r\in \mathbb Q$ such that $r>\sqrt 2$, we have $\forall
     x\, (x>0 \rightarrow (x,rx)\notin \Psq)$.
    \item  For all $x\neq 0$,  the set
    $$\{y:(x,y)\in \Psq \}$$ is closed downwards, and has no supremum.
    Furthermore, $$\mbox{Inf }\{y_2-y_1: (x,y_1)\in \Psq \,\& \, (x,y_2)\notin \Psq\}\,
    =\, 0.$$
    \item An axiom expressing the fact that the composition of $x\mapsto \sqrt 2 x$ with itself
    yields the map $x\mapsto 2x$:
   \[\forall (x,y>0)\left([(\exists z>0)\Psq(x,z)\land \Psq(z,y)]\iff y<2x\right).\]
    \item The quantifier-free theory of $\Q2$.
%
%
\end{enumerate}
%
Clearly, $\Q2$ is a model of $T$.

 For simplicity we write $F=\Qq({\sqrt 2})$.  Before we prove quantifier
elimination we note that if $\CM$ is a model of $T$ then we may consider the
associated $F$-vector space $V=F\otimes_{\mathbb Q}M$. If we identify $M$ with the
$\mathbb Q$-subspace $1\otimes M$, then  each element of $V$ can be written uniquely
as $x+{\sqrt 2}y$ for $x,y\in M$. We can now endow $V$ with an ordering by declaring
$x+{\sqrt 2}y>0$ when $(y,-x)\in \Psq$. Indeed, the above axioms imply that this is
a linear ordering of the vector space $V$, compatible with the ordering of $F$.

The definition of the ordering and Axiom (3) allows us to conclude:
\begin{claim} \begin{enumerate}
\item For $x,y\in M$, we have $(x,y)\in \Psq \Leftrightarrow$ if and only if  $y<
{\sqrt 2} x$ in $V$. \item $M$ is dense in $V$.\end{enumerate}
\end{claim}
We can now endow $V$ with an $\CL_{\sqrt 2}$-structure, by interpreting $\Psq$ as we
did over $\mathbb Q$. Clause (1) above then implies that $\CM$ is a substructure of
$V$ as an $\CL_{\sqrt 2}$-structures.

 The following lemma is similar to
\cite[Proposition 3.3]{ElHaKe}:

\begin{lemma} \label{QE}
    The theory $T$ is complete and has quantifier elimination.
\end{lemma}
\proof
    Let $\CQ_1, \CQ_2\models T$ be $\kappa$-saturated models of the same cardinality.
In order to prove quantifier elimination it suffices to prove (see for example
\cite[Corollary 3.1.6]{Marker-book}):

If $A$ is a  substructure of $\CQ_1$ and $\CQ_2$ of cardinality smaller than
$\kappa$, then for every $a_1\in Q_1$ there is $a_2\in Q_2$ such that $a_1$ and
$a_2$ have the same quantifier-free type over $A$.

As above, consider the ordered $F$-vector spaces $\CG_i:=F\otimes_{\mathbb Q}
\CQ_i$. Since $Q_i$ is dense in $G_i$, and $\CG_i$ is o-minimal, the saturation of
$\CQ_i$ implies that $\CG_i$ is also
     $\kappa$-saturated.  Let $B_i$ be the $F$-span of $A$ inside $\CG_i$. Then
     $B_1$ and $B_2$ are isomorphic-over-$A$ ordered vector spaces (both isomorphic to
     $A+{\sqrt 2}A$, with the same ordering). Thus we may write $B=B_1=B_2$

    Let $p(x):=\tp_{\CG_1}(a_1/B)$. We may assume that $a_1\notin A$ and hence
    $a_1\notin B$ (note that $B\cap Q_1=A$).
    By the completeness of the theory of ordered $F$-vector
    spaces and saturation, we can find $a_2\in \CG_2$ such that $a_2\models p(x)$. In fact, because $\CG_2$
    is $\kappa$-saturated and $p$ is non-algebraic there is more than one such
    $a_2$, so since $Q_2$ is dense in $G_2$, we can find such an $a_2$ inside $Q_2$.

Finally, since each $\CQ_i$ is a substructure of $\CG_i$, and $a_1,a_2\models p$, it
follows that the quantifier-free types of $a_1$ and $a_2$ over $A$, in the
structures $\CQ_1$, $\CQ_2$, respectively, are the same. This completes the proof of
quantifier elimination.

To see that $T$ is complete we just  notice that every model of $T$ contains the
structure $\Q2$, which is itself a model of $T$.\qed

\begin{corollary}
    The theory $T$ is a tight weakly o-minimal non-valuational theory and
    $T^*$ is the theory of ordered $\Qq^{\sqrt{2}}$-ordered vector spaces (in the language $\CL_{\sqrt 2}$).
\end{corollary}
\begin{proof}
    The atomic subsets of $Q$, the universe of any $\CQ\models T$, are rays with or without endpoints.
    By quantifier elimination the definable subsets of $Q$ are in the boolean algebra generated by
    those, proving the
    weak o-minimality. The same argument also shows that the only definable cuts are non-valuational,
    because so are the atomic cuts.

By the proof of lemma \ref{QE} and the preceding discussion, each model $\CQ$ of $T$
is a dense substructure of the o-minimal structure $V=F\otimes_{\mathbb Q} Q$. It is
easy to verify that the intersection with $Q$ of every ray $(-\infty,a)$ in $V$ is
definable in $\CQ$, and hence every element of $V$ realizes a definable cut in
$\CQ$. Conversely, by quantifier elimination, the definable cuts in any model
$\mathcal Q$ of $T$ are of the form $x_1+r\sqrt{2}x_2$ for $r\in \Qq$ and
$x_1,x_2\in Q$, so they are realized in $V$. It follows that $V$ is the canonical
completion of $\CQ$, and its theory, in the language $\CL_{\sqrt 2}$, is that of an
ordered $F$-vector space.

To see that $T$ is tight we note that each definable cut in $\CQ$ can also be
written as $x_1+\sqrt{2}x_2$, for $x_1,x_2\in Q$, and that this is a definable
family in  $T$.
\end{proof}

Note that the above construction worked because of the algebraicity of $\sqrt 2$. If
we consider
 $\Qq^t$,  the expansion of $\Qq_{vs}$ by $x\mapsto tx$ where $t$ realizes  a cut defining
 a real transcendental
  number we would not obtain a tight structure. See the example
  $\Qq^\pi_{vs}$ in \cite{ElHaKe}.

We now prove:
\begin{theorem}
    The structure $\Q2$ is not elementarily equivalent to a reduct of an o-minimal trace.
\end{theorem}
\proof
    Assume towards a contradiction that there is a dense pair $(\CR,\CQ)$ of o-minimal
    expansions of groups such that $\Q2$ is elementarily equivalent to a reduct of the
    trace which this pair induces on a structure $\CQ$. By that we mean that there is some
    expansion $\hat {\mathcal Q}$ of $\la \mathcal Q,<,+ \ra$
    satisfying $T$, such that every definable set in $\hat {\mathcal Q}$ is definable
    in the dense pair $(\CR,\mathcal Q)$. While some of these sets are
    already definable in the o-minimal structure $\CQ$ others may be the intersection with $Q^n$ of
    subsets of $R^n$
    that are definable in  $\CR$ over parameters which are not in $\CQ$. The order relation $<$
    and the group operation $+$ are assumed to be
    definable in $\CQ$.

Let us consider the predicate $\Psq(\hat{\mathcal Q})$. It is a definable set in
$(\CR,\CQ)$, hence by \cite[Theorem2]{vdDriesDense}, there is a definable $Y_{\sqrt
2}\sub R^2$ in the o-minimal structure $\CR$ such that $Y_{\sqrt 2}\cap \mathcal
Q=\mathcal \Psq(\hat{\mathcal Q})$. Because $\mathcal Q$ is dense in $R$ (the
universe of the o-minimal structure $\CR$), it easily follows that for every $x\in
Q$, there is $y(x)\in R$ such that
$$y(x)=sup\{y\in \mathcal Q:(x,y)\in
Y_{\sqrt 2}\}.$$ By taking the closure of the graph of $y(x)$ we obtain an
$\CR$-definable function, which we will denote by $\lambda_{\sqrt 2}:R\to R$, which
gives $y(x)$ for every $x\in \mathcal Q$. It is not hard to see that $\lambda_{\sqrt
2}$ is a definable automorphism of $\la R,+\ra$ satisfying $\lambda_{\sqrt
2}\circ \lambda_{\sqrt 2}(x)=2x.$

We now  consider two cases. If the function $\lambda_{\sqrt 2}$ is $\0$-definable in
$\CR$ then it comes from a definable function in the o-minimal structure $\mathcal
Q$, and in particular, for every $x\in \mathcal Q$, the set $\{y\in \mathcal
Q:(x,y)\in \Psq\}$ has a supremum in $\mathcal Q$. This contradicts the axioms of
$T$.

On the other hand, if $\lambda_{\sqrt 2}$ is not $\0$-definable then by
\cite{MilStar}, one can define in the o-minimal structure $\CR$ a multiplication
function $\cdot$ on $R^2$, making $\la R,<,+,\cdot\ra$ a real closed field, call it
$K$. A-priori the multiplication function might not be $\0$-definable but in that
case there is a $\0$-definable family of such multiplications all of which expand
$\la R,+\ra$ to   a real closed field. By definable choice we may find one such
multiplication function that is $\0$-definable.

 Since $\lambda_{\sqrt 2}$ is
an $\CR$-definable automorphism of the additive group of $K$ it must be of the form
$x\mapsto c\cdot x$ for some scalar $c\in K$. Because $\lambda_{\sqrt 2}\circ
\lambda_{\sqrt 2}(x)=2x$, and because $\lambda_{\sqrt 2}$ takes positive values on
$x>0$, the scalar $c$ is necessarily $\sqrt 2$ (in the sense of $K$). In particular,
$\lambda_{\sqrt 2}$ is $\0$-definable in $\CR$, yielding a contradiction as
before.\qed


\section{The theory of $(\CM^*,\CM)$}  \label{pairs}

  From now on, given a
complete non-valuational theory $T$ we will denote by $T^*$ the theory of the
associated o-minimal completion, in the language $\CL^*_\0$ (by Corollary
\ref{equiv}, the theory $T$ indeed determines $T^*$). We write $\bar \CM$ and
$\CM^*$,  for the structure $\bar \CM_\0$, and $\CM^*_\0$, respectively.

While $\CM$ and $\CM^*_\0$ initially have different signatures it will be convenient
to treat them in the same langauge. We thus modify the language of $\CM$.

\begin{lemma}
    Let $\CM$ be a weakly o-minimal non-valuational structure. Let $\CM_0$ be the reduct of $\CM$
    generated by all $\0$-definable closed sets. Then every $\0$-definable set in  $\CM$
    is $\0$-definable in $\CM_0$. In particular, $\CM$ and $\CM_0$ have the same $\0$-definable sets.
\end{lemma}
\begin{proof}
    This follows from the proof of Proposition \ref{observation}.
%
%
%
\end{proof}

    So from now on we will assume that $\CM$ is given in the signature consisting of
    a function symbol for $+$, the ordering $<$, and a predicate for each
    $\0$-definable closed set in $M^n$. We let $\CL$ be the associated language, so
    we may use the same language for $\CM^*$. By Proposition
    \ref{QE}, the structure $\CM^*$ eliminates quantifiers.

 We let $\CL^P=\CL\cup \{P\}$, where $P$ is a unary
     predicate. We consider the $\CL^P$-structure
     $$\CM^P=\la \CM^*, \CM\ra,$$ where the interpretation of $P$ is
     $M$. As we will see, the theory of $\CM^P$ depends only on $T$.
     We propose the following axiomatization for  this theory:

Let $T^d$ be the $\CL^P$-language axiomatized as follows (we write
$(\CM',\CM)$ for models of $T^d$),

\begin{enumerate}
\item $\CM \models T$, $\CM'\models T^*$.

\item  $M$ dense in $M'$.

\item  Every definable cut in $\CM$ has a supremum in $M'$.

\item (when $T$ is tight) Every element of $\CM'$ realizes a
definable cut in $\CM$.

\end{enumerate}

Our goal is to prove:
\begin{theorem}\label{axioms} The theory $T^d$ is complete.\end{theorem}

\subsection{The tight case}

Assume that $T$ is tight. As we saw in Proposition \ref{tight1},
the structure $\CM^*$ is interpretable in $\CM$ without
parameters. Using axiom (4) above we immediately conclude:

\begin{lemma}\label{pairs-tight}
Assume that $T$ is tight. \begin{enumerate} \item If
$(\CM',\CM)\models T^d$ then necessarily $\CM'=\CM^*$.

\item   For all $\CM, \CN\models T$,
    we have $(\CN^*, \CN)\equiv (\CM^*,\CM)$.

    \end{enumerate}
\end{lemma}

\subsection{The general case}

\begin{theorem}\label{pairs-main} If $\CM^d=(\CM',\CM)$ and $\CN^d=(\CN',\CN)$ are models of
$T^d$, then $\CM^d\equiv \CN^d$.\end{theorem}

\proof We may assume that $T$ is non-tight. We may assume that $\CM^d$ and $\CN^d$
are $\kappa$-saturated for sufficiently large $\kappa$.

Notice that every $\CM$-definable cut is realized in $\CM'$ exactly once, hence
there is a natural embedding of $\CM^*$ into $\CM'$, and the same holds for $\CN'$
and $\CN$. However, by saturation, unless $\CM$ is tight it is not the case that
$\CM'$ {\bf equals} $\CM^*$, since it realizes cuts which are not definable as well.
Our goal is to show that there are $(B,A)\prec (\CM',\CM)$ and $(D,C)\prec
(\CN',\CN)$ which are isomorphic.

Notice first that  both $M$ and $M'\setminus M$ are dense in $M'$,
for $i=1,2$. Indeed, this follows from the fact that $T$ is
non-valuational, so if $c\in \bar M\setminus M$ is any element
then $c+M\sub \bar M$ is dense in $\bar M$, so also in $M'$.

Since $\CM^*\models T^*$ and $\CM^*$ eliminates quantifiers, the
pair $(\CM',\CM^*)$ is an elementary dense pair of o-minimal
structures, so we shall apply to it the theory of dense pairs as
in \cite{vdDriesDense}.

We first need:
\begin{lemma}\label{dcl}
Let $(\CM',\CM)\models T^d$. Let $M_0\prec \CM$. Then $\dcl_{\CM_0^*}(M_0)=\bar \CM_0$. Moreover,  $\dcl_{\CM'}(M_0)=\bar M_0$.
\end{lemma}
\proof 
It will suffice to prove the first part of the lemma as the second part follows from the fact that $\CM_0^*\prec \CM'$.


 First we show the right-to-left inclusion. For that we need:

\begin{claim} Assume that $f:M_0^n\to \bar M_0$ is a $\0$-definable function in $\CM_0$.
Then there are in $\CM_0$ finitely many $\0$-definable strong cells
of the form $C_1,\ldots, C_k\sub M_0^n$, with $M_0^n \sub \bigcup_i
\bar C_i$, and in $\CM_0^*$ there are finitely many $\0$-definable
functions $\bar f_i: \bar C_i\to \bar M_0$, such that for all $x\in
C_i$, $\bar f_i(x)=f(x)$.
\end{claim}
\proof We decompose $M_0^n$ into $\0$-definable strong cells,
$C_1,\ldots, C_k$, on each of which $f$ is strongly continuous.
For each $i$, the graph of $\bar f\upharpoonright \bar C_i$ is the
iterative convex hull of $\Gamma(f\upharpoonright C_i)$, so it is
$\0$-definable in $\CM^*_0$.\qed

Assume now that  $b\in \bar M_0$, then by definition of the completion, the cut
$Y=\{x\in M_0:x<b\}$ is definable in $\CM_0$, over a tuple of parameters $a$. We may
assume that $Y=Y_a$ for a $\0$-definable family of sets $\{Y_t:t\in T\}$  and
$\0$-definable set $T\sub M_0^m$, and that we have $b=\sup\, Y_a$. It follows that
there is in $\CM_0$ a $\0$-definable function $f:T\to \bar M_0$, such that $f(a)=b$.

By the above claim, we have $T=\bigcup C_i$ a union of
$\0$-definable strong cells in $\CM$, and there are $f_i:\bar C_i\to \bar M_0$ all $\0$-definable
 in $\CM^*_0$, , such that
\begin{equation}\label{eq:1} \bigwedge_{i=1}^k \forall x\in C_i \,\, \bar f(x)=f(x).\end{equation}  In particular, there
is $i\in \{1,\ldots, k\}$ such that $a\in C_i$ and $b=\bar f_i(a)$ is in
$\dcl_{\CM^*_0}(M_0)$. Thus, $\bar M_0\sub \dcl_{\CM^*_0}(M_0)$.

For the converse, we assume that $g(a)=b$ for some $\0$-definable
function $g$ in $\CM^*_0$ and $a\in M^m_0$. We want to show that $b\in
\bar M_0$, namely that $b$ is the supremum of a definable cut in the
structure $\CM_0$.

 The function $g$ is $\0$-definable in the o-minimal structure
$\CM^*_0$, so by Theorem \ref{intersection}, the set
$$Y=\{(x,y)\in M^{n+1}_0: y<g(x)\}$$ is $\0$-definable in $\CM_0$ and we have
\begin{equation} \label{eq:2} \forall x\in M_0^n\, \, g(x)= \sup(Y_x).\end{equation}
 It follows that $b=g(a)=\sup Y_a$, with $Y\sub M^{n+1}_0$ a
$\0$-definable set in $\CM_0$. Hence, $b\in \bar M_0$.\qed

We will also need:
\begin{claim} \label{types} For $A\sub M$ and $a\in M$, the $\CM$-type of $a$ over $A$
is determined by the cut of $a$ in $\dcl_{\CM'}(A)$.\end{claim}

\proof Assume that $a$ and $b$ in $M$ realize the same cut over $\dcl_{\CM'}(A)$. To
see that $a$ and $b$ realize the same $\CM$-type over $A$, it is sufficient, by the
weak o-minimality of $\CM$,  to show, for every cut $C\sub M$ definable in $\CM$
over $A$, that $a\in C$ iff $b \in C$ . Using our assumptions, it is enough to prove
that the supremum of $C$ exists in $M'$ and belongs to $\dcl_{\CM'}(A)$.

 If
$C$ has a supremum $s$ in $M$ then $s\in \dcl_{\CM}(A)\cap M$, and
therefore (Lemma \ref{early-dcl}) $s\in \dcl_{\CM^*}(A)$. Since
$\CM^*$ is an elementary substructure of $\CM'$ we have $s\in
\dcl_{\CM'}(A)$.

 If $C$ has no supremum in $M$ then, by definition, its supremum is realized in
$\bar M$. As $C$ is definable in $\CM$ over $A$, its closure in $\bar M$ is
$\0$-definable in $\CM^*_A$, so by \ref{QE}(2) it  is definable in $\CM_\0^*$ over
$A$. But then $\sup C\in \dcl_{\CM^*}(A)=\dcl_{\CM'}(A)$. This finishes the
proof.\qed

The rest of the proof follows closely the arguments from  \cite{vdDriesDense}. In
order to proceed we borrow the following terminology:

\begin{definition}
    For $B\sub M'$ and $A=B\cap M$, we say that $(B,A)$ is {\em free}
if $\dim_{\CM'}(B'/A)=\dim_{\CM'}(B'/M)$ for every finite $B'\sub B$. Namely, every
subset of $B$ which is $\CM'$-independent over $A$ remains independent over $M$. We make the same definitions for subsets of $N'$ and $N$.
\end{definition}

We consider all $(B,A)\sub (M',M)$ (and similarly $(D,C)$ in $(N',N)$) which
satisfy:

\noindent (i) $B\cap M=A$.

\noindent (ii) $\dcl_{\CM'}(B)=B$.

\noindent (iii) $(B,A)$ is free.

We now begin the construction of the intended isomorphism.  By saturation, there is
$\CM_0\prec \CM$, of cardinality smaller than $\kappa$ that is isomorphic to some $\CN_0\prec
\CN$.

If we let $A_0:=M_0$ $B_0:=\bar A_0$ and $C_0=:=\bar N_0$, $D_0:=\bar C_0$. Then (i) holds. By Lemma \ref{dcl}
$\dim_{\CM'}(B_0/A_0)=0$, so $(B_0,A_0)$ is (trivially) free. Also, by this lemma,
$B_0$ is definably closed in $\CM'$, so $(B_0,A_0)$ satisfy (i),(ii),(iii).
Similarly, $(D_0,C_0)$ satisfies (i),(ii),(iii).

Our goal is to use back-and-forth and Tarski-Vaught in order to build isomorphic
elementary substructures of $(\CM',\CM)$ and $(\CN',\CN)$. Towards that goal we need
to prove the following result:

\begin{lemma}\label{iso-extend} Assume that $(B,A)\sub (\CM',\CM)$ and $(D,C)\sub (\CN',\CN)$ satisfy (i),(ii),(iii), and isomorphic (namely, there
is an $\CL$-isomorphism $\alpha:B\to D$ sending $A$ onto $C$), with $|A|<\kappa$.
Then, for every $b\in M'$, there are $B'\sub M', A'\sub M$ with $b\in B'$, and there
are  $D'\sub N', C'\sub N$, such that  $(B',A')\, ,\, (D',C')$ satisfy
(i),(ii),(iii), and there is an isomorphism $\alpha':(B',A')\to (D',C')$ extending
$\alpha$.
\end{lemma}
(We also have the analogous result for $(D,C)$ and $d\in N'$.)

\proof We divide the argument into several cases:
\\

\noindent {\bf Case I.} $b\in M$. \vspace{.3cm}

First, we find $d\in N$ such that $\alpha(\tp_{\CM'}(b/B))=\tp_{\CN'}(d/D)$ (so by
Lemma \ref{types}, also $\alpha(\tp_{\CM}(b/A))=\tp_{\CN}(d/C)$). Indeed, this is
possible because $N$ is dense in $N'$ and $\CN'$ is $\kappa$-saturated. The function
$\alpha$ then extends naturally to an isomorphism $\alpha'$ of the o-minimal
structures $B':=\dcl_{\CM'}(Bb)$ and $D':=\dcl_{\CN'}(Dd)$. We let
$A'=B'\cap M$ and $C'=D'\cap N$. In order to see that $\alpha'$
is an isomorphism of $(B',A')$ and $(D',C')$ it is left to verify is that for every
$a\in B'$,

\begin{equation}\label{eq:iso} a\in M \Leftrightarrow \alpha'(a)\in N.\end{equation}

So, we take $a\in \dcl_{\CM'}(Bb)$ and prove (\ref{eq:iso}).

Assume first that $a\in \dcl_{\CM'}(Ab)$. By Lemma \ref{dcl}, $a\in \bar M$, so we
have $a\in \dcl_{\CM^*}(Bb)$. Hence, there exists a $\0$-definable function $F$ of
$(n+1)$-variables  in $\CM^*$, and $e\in (\bar M)^n$, with $F(b,e)=a$. The function
$F$ is definable in $\CM^*$, and, by \ref{intersection},  its restriction to
$M^{n+1}$ is $\0$-definable in $\CM$ (as a function into $\bar M$). Thus, we can
definably in $\CM$ partition its domain into $\0$-definable strong cells on each of
which $F$ takes either values in $M$ or in $\bar M\setminus M$. This partition is
part of the weakly o-minimal theory $T$, and thus holds in both $\CM$ and $\CN$.
Since $\alpha(\tp_{\CM}(b/A))=tp_{\CN}(d/C)$ it follows that $a=F(b,e)\in M$ if and
only if $\alpha'(a)=F(d,\alpha(e))\in N$.

Assume now that $a\in \dcl_{\CM'}(Bb)\setminus \dcl_{\CM'}(Ab)$ (so $\alpha'(a)\in
\dcl_{\CN'}(Dd)\setminus \dcl_{\CN'}(Cd)$). We claim that $a\notin M$ and
$\alpha'(a)\notin N$.

Indeed, assume towards a contradiction that $a\in M$, and let $Y\sub B$ be a minimal
finite set which is $\dcl_{\CM'}$-independent over $Ab$ such that $a\in
\dcl_{\CM'}(YAb)$. Because $a\notin \dcl_{\CM'}(Ab)$ the set $Y$ is nonempty so fix
$y_0\in Y$. We have $a\in \dcl_{\CM'}(Y'y_0Ab)$, with $Y'=Y\setminus \{y_0\}$, so by
exchange (and minimality of $Y'$), $y_0\in \dcl_{\CM'}(Y'Aba)$. Because $a,b\in M$
and $A\sub M$, it follows that $Y$ is not independent over $M$, even though it is
independent over $A$. This contradicts the fact that $(B,A)$ was free, so $a\notin M$. The same argument shows that $\alpha'(a)\notin N$.

Thus, we showed that $\alpha':(B',A')\to (D',C')$ is an isomorphism. It is clear, that the pairs satisfy (i) and (ii), so we are left to see that they are free.
So, we take $Y\sub B'$ independent over $A'$ and claim that it remains independent
over $M$. Indeed, because $b\in A'$ (since $b\in M$), it must be the case that
$Y\sub B$, and the result follows immediately from the freeness of $(B,A)$ (because
$A\sub A'$). This ends Case I. \vspace{.3cm}

\noindent{\bf Case II.} $b\in \dcl_{\CM'}(BM)$. \vspace{.3cm}

In this case, there is $\bar m=(m_1,\ldots, m_k)\in M^k$ such that $b\in
\dcl_{\CM'}(B\bar m)$. We first apply Case I to each $m_i$, and thus may assume that
$\bar m\sub B$, and in particular may assume that $b$ is already in
$B$.\vspace{.3cm}

\noindent {\bf Case III.} $b\notin \dcl_{\CM'}(BM)$\vspace{.3cm}

Notice first that in this case $\CM$ (and hence also $\CN$)  is not tight (since in
the tight case $M'=\bar M=\dcl_{\CM'}(M)$). We let $B'=\dcl_{\CM'}(Bb)$ and
$A'=B'\cap M$. Our goal is to show that $(B',A')$ satisfies (i),(ii),(iii), so we
need to show that it is free.

We first claim that $A'=A$. Indeed, if $a\in \dcl_{\CM'}(Bb)\cap M$ then either $a\in
\dcl_{\CM'}(B)$, so $a\in A$, or if not then by exchange, $b\in \dcl_{\CM'}(Ba)$,
contradicting the assumption on $b$.

Assume now that $Y\sub B'$ is independent over $A'=A$. If $Y\sub B$ then $Y$ is
independent over $M$, and otherwise, we may assume that it is of the form $Y'b$ with
$Y'\sub B$. By freeness of $(B,A)$ we have $Y'$ independent over $M$ and by
assumption on $b$ we may conclude that  $Y'b$ independent over $M$. Thus, $(B',A')$
is indeed free.

Next, we claim that we may find in $\CN'$ an element $d$ such that
$\alpha(\tp_{\CM'}(b/B))=\tp_{\CN'}(d/D)$  and in addition $d\notin
\dcl_{\CN'}{(DN)}$. It is here that we use the fact that $\CN$ is non-tight. We
prove:

\begin{lemma}\label{nontight} Let $D\sub N'$ be of cardinality smaller than $\kappa$.
Then for every $\CM'$-type $p(x)$ over $B$, there is a realization
of $\alpha(p)$ which is not in $\dcl_{\CN'}(DN)$.
\end{lemma}
\proof

By the saturation of $(N',N)$ it is sufficient to prove that
$X\nsubseteq \dcl_{\CN'}(DN)$ for every infinite set $X\sub N'$
that is definable in $\CN'$ over $D$. For that it is clearly
sufficient to show that $X\nsubseteq \dcl_{\CN'}(D\bar N)$. By
applying the theory of dense pairs to the pair of o-minimal
structures $(\CN',\bar \CN)$, we may conclude from \cite[Lemma
4.1]{vdDriesDense}, that no interval in $\CN'$ is in the image of
$\bar N^n$ under an $\CN'$-definable map. This is easily seen to
imply the result we want. \qed

This ends the proof of Lemma \ref{iso-extend}.\qed

Going back to our proof of completeness of $T^d$, we find $d\in N'$ with
$\alpha(\tp_{\CM'}(b/B))=\tp_{\CN'}(d/D)$ and with $d\notin \dcl_{\CN'}{DN}$. We let
$D'=\dcl_{\CN'}(Dd)$ and $C'=D'\cap N'$ (which equals $C$), so as before $(D',C')$ is
free. It is left to see that the natural extension of $\alpha$ to $\alpha':B'\to D'$
preserves $M\cap B'$. However, $B'\cap M=A'$ so by applying what we already know to
both $\alpha$ and $\alpha^{-1}$ we conclude that $x\in M' \Leftrightarrow
\alpha'(x)\in N'$. This ends the proof of Theorem \ref{pairs-main}.\qed

Notice that the proof above showed that any isomorphism of weakly o-minimal
structures $M_1\prec M$ and $M_2\prec N$ can be extended to an isomorphism of
elementary substructures $(B,A)\prec (M',M)$ and $(D,C)\prec (N',N)$. Lemma
\ref{iso-extend} also implies:
\begin{lemma}\label{more} Assume that $(\CM',\CM), (\CN',\CN)\models T^d$ and $(B,A)\sub
(\CM',\CM)$, $(D,C)\sub (\CN',\CN)$ satisfy (i),(ii),(iii). If $\alpha:B\to D$ is an
$\CL$-isomorphism sending $A$ to $C$ and $\alpha(b)= d$ for some $b\in B^n$ then
$$\alpha(\tp_{(\CM',\CM)}(b/\0))=\tp_{(\CN',\CN)}(d/\0).$$
\end{lemma}

We can now prove analogues of several theorems from \cite{vdDriesDense}. The proofs
are very similar to the original ones.

\begin{theorem}\label{stable-embed} Let$\CM^d=(\CM',\CM)$ be a model of $T^d$.
\begin{enumerate}

\item In $\CM^d$, every $\0$-definable subset of $(M')^n$ is a boolean combination
of sets defined by formulas of the form \begin{equation}\label{formula} \exists
x_1\cdots \exists x_k (\bigwedge_{i=1}^k x_i\in P \, \&\, \phi(x_1,\ldots,
x_k,y),\end{equation} where $|y|=n$ and $\phi(x,y)$ is an $\CL$ formula.

\item Let $B\sub M'$ be such that $(B,B\cap M)$ is free. Then every subset of $M^k$
that is definable in $\CM^d$ over $B\sub M'$ is of the form $Y\cap M^k$ for some
$Y\sub (M')^k$ that is definable in $\CM'$ over $B$.

\item Every subset of $M^k$ that is definable in $\CM^d$ over $A_0\sub M$ is
definable in the structure $\CM$ over $A_0$.

 \item Every subset of $M^n$ that is definable
in $(\CM^*,\CM)$ (here $\CM^*$ {\bf is} the completion of $\CM$) is definable in the
structure $\CM$.

\end{enumerate}
\end{theorem}
\proof  Without loss of generality, $(\CM',\CM)$ is sufficiently saturated.

(1) By standard model theoretic considerations it is enough to prove the following:
For any $b,d\in (M')^k$, assume that $b$ satisfies a formula of the form
(\ref{formula}) if and only if $d$ does. Then $b$ and $d$ have the same type in
$\CM^*$ over $\0$.

Let $r=\dim_{\CM'}(b/M)$. We can find $a\subseteq M$ finite such that
$\dim_{\CM'}(b/a)=r$. It follows that if we let $B=\dcl_{\CM'}(ab)$ and
$A=\dcl_{\CM'}(a)$ then $(B,A)$ is free and $A=B\cap M$.

We consider the $\CL$-type of $(b,a)$ over $\0$. Because $b$ and $d$ realize the
same formulas of the form (\ref{formula}), and because of saturation we can find
$c\in M$ such that $\tp_{\CM'}(b,a/\0)=\tp_{\CM'}(d,c/\0)$. The pair $(D,C)$, with
$D=\dcl_{\CM'}(cd)$ and $C=\dcl_{\CM'}(c)$ is free with $C=D\cap M$. Just like in the
proof of Lemma \ref{iso-extend}, the natural $\CL$-isomorphism of $B$ and $D$ (sending
$(b,a)$ to $(d,c)$) sends $A$ to $C$.

By Lemma \ref{more}, the $\CL^P$-types of $b$ and $d$ in $(\CM',\CM)$ are the same.
Thus we proved (1).

(2) By standard model theoretic arguments it is sufficient to prove: If $b_1,b_2\in
M^k$  satisfy the same $\CM'$-type over $B$ then they satisfy the same $L^P$-type
over $B$. For that, let $A=B\cap M$. It is sufficient to show that there are
$(B_1,A_1), (B_2,A_2)\prec  (\CM',\CM)$, with $(B,A)\sub (B_i,A_i)$ and $b_i\in B_i$
for $i=1,2$, and there is an $\CL$-isomorphism between $(B_1,A_1)$ and $(B_2,A_2)$,
which fixes $B$ point-wise, and sending $b_1$ to $b_2$.


We are now in the setting of Case I of the proof of Lemma \ref{iso-extend}, with our
$b_1,b_2$ replacing $b,d$ there. Thus, we may first find two free pairs
$(B_1',A_1')$ and $(B_2',A_2')$ with $B\sub B_i'$ and $b_i\in B_i'$, $i=1,2$, and an
isomorphism $\alpha:(B_1',A_1')\to (B_2',A_2')$ extending the identity map, with
$\alpha(b_1)=b_2$. We now proceed exactly as in the proof of Theorem
\ref{pairs-main} and obtain the desired $(B_1,A_1), (B_2,A_2) \prec (\CM',\CM)$.
Thus, $b_1$ and $b_2$ realize the same $\CL^P$ type over $B$ and we may conclude
(1).

For (3), let $X\sub M^k$ be definable in $(\CM',\CM)$ over $A_0\sub M$. Notice that
the mere definability of $X$ in $\CM$ follows immediately from (2) but we want to
show that $X$ is definable over the same $A_0$. For that, it is sufficient to prove
that any $a_1,a_2\in M$ which realize the same $\CM$-type over $A_0$ realize the
same $\CL^P$-type over $A_0$.

To do that, we first find a small model $\CM_1\prec \CM$ containing $A_0$ $a_1,a_2$,
and an automorphism $\alpha$ of $\CM_1$ over $A_0$, sending $a_1$ to $a_2$. As we
commented previously, we may the extend $\alpha$ to an isomorphism of two structures
$(B,A), (D,C) \prec (\CM',\CM)$. This is clearly sufficient.

 To see (4), we note that every element of $\CM^*$ is in $\dcl_{\CM^*}(N)$ and hence
every definable subset of $\bar M^k$ in $\CM^*$ can be defined over $M$. We now
apply (3).\qed\\

Note that (3) above fails if we omit the requirement that $M_0\sub M$, since in the
non-tight case, in general,  $\CM'$ will realize cuts which are not definable in
$\CM$ and thus their intersection with $M$ is not definable in $\CM$.

We also point out:
\begin{lemma}\label{definably complete} If $\CM^d=(\CM',\CM) \models T^d$ then it is
definably complete.
\end{lemma}
\proof If $X\sub M'$ is definable in $\CM^d$ and bounded below then the intersection
of its convex hull with $M$ is definable in $\CM^d$, and thus has the form $Y\cap M$
for some $Y\sub M'$ which is definable in $\CM'$. Without loss of generality, $Y$ is
also convex and thus $\mbox{Inf} Y=\mbox{Inf} X$. This suffices, by o-minimality of $\CM'$.\qed \\

We can now conclude, using Boxall and Hieronymi, \cite{BH}:

\begin{theorem}
    Let $\CM^d=(\CM',\CM)\models T^d$. If $U\sub (M')^n$ is open and
    definable in $\CM^d$ then it is definable in $\CM'$. More precisely, if an open $U$ is defined in
    $\CM^d$ over $B\sub M'$ such that $(B,B\cap M)$ is free,
    then $U$ is definable in $\CM'$ over $B$. In particular, $\CM^P$ has an
    o-minimal open core.
\end{theorem}
\begin{proof} This is an immediate corollary of \cite[Corollary 3.2]{BH} and what we
proved so far. We extract from their argument a direct proof, which is underlined by
the following simple corollary of cell decomposition.

\begin{fact}\label{empty interior} If $Y\sub (M')^n$ is definable in $\CM'$ and $\dim Y<n$ then $Y\cap
M^n$ has empty interior in $M^n$.
\end{fact}

We now first claim that $\cl_{M'}(U)$ is definable in $\CM'$ over $B$. Indeed, by
Theorem \ref{stable-embed} (2), there is $Y\sub (M')^n$ definable in $\CM'$ over $B$
such that $Y\cap M^n=U\cap M^n$. By the above observation, $\dim Y=n$.

Since $M^n$ is dense in $(M')^n$, the set $Int(Y)\cap M^n$ is dense in the open set
$Int(Y)$. We claim that it is also dense in $U$. Indeed, we know that  $Y\cap
M^n=U\cap M^n$ is open in $M^n$ and dense in $U$, and  by o-minimality
$\dim_{\CM'}(Y\setminus Int(Y))<n$.  It thus follows from Fact \ref{empty interior},
that $Int(Y)\cap M^n$ is dense in $U$.

So, $$\cl_{M'}(U)=\cl_{M'}(Int(Y)\cap M^n)=\cl_{M'}(Int(Y)).$$ Because $Y$ was
definable in $\CM'$ over $B$, $\cl_{M'}(U)$ is definable in $\CM'$ over $B$.

We thus showed that the closure of every $\CM^d$-definable open set over $B\sub M'$
is definable in $\CM'$ over $B$. It follows that every $\CM^d$-definable continuous
function $f:(M')^n\to M$ is definable in $\CM'$, over the same parameters. Indeed,
the closure of the open set $\{(x,y)\in (M')^{n+1}:y<f(x)\}$ is exactly $\{(x,y)\in
(M')^{n+1}:y\leq f(x)\}$, from which the definability of $f$ follows.

Finally, we show that every closed $F\sub (M')^n$ set which is $\CM^d$-definable
over $B\sub M'$ is definable in $\CM'$ over $B$. For every $x\in M^n$ we let
$f(x)=d(x,F)=Inf\{d(x,y):y\in F\}$. By Lemma \ref{definably complete}, this is a
well defined function in $\CM^d$ (over $B$), and since $F$ is closed, the function
$f$ is continuous and $F$ is its zero set. Because $f$ is definable in $\CM'$ over
$B$, so is the set $F$.

Since every definable set in $\CM'$ can be defined over some
$B\sub M'$ with $(B,B\cap M)$ free, the theorem follows.
\end{proof}

\end{document}